\documentclass[12pt,psfig]{article}
\usepackage{multirow}
\usepackage{longtable}
\usepackage{bibentry}
\usepackage{amsfonts}
\usepackage{amssymb}
\usepackage{amsthm}
\usepackage[namelimits]{amsmath}
\usepackage{graphicx,epsfig}
\usepackage{amscd}
\usepackage{epsfig,color}
\usepackage{tikz}
\usepackage{pgflibraryarrows}
\usepackage{pgflibrarysnakes}

\usepackage{subfig}

\usepackage{tikz}
\usetikzlibrary{positioning}

\date{}

\newcommand{\de}{\delta}

\newcommand{\ep}{\epsilon}
\newcommand{\Ga}{\Gamma}
\newcommand{\ga}{\gamma}

\newcommand{\al}{\alpha}
\newcommand{\be}{\beta}

\newcommand{\si}{\sigma}
\newcommand{\ld}{\lambda}
\newcommand{\Ld}{\Lambda}
\newcommand{\ba}{\begin{array}}
\newcommand{\ea}{\end{array}}
\newcommand{\beqq}{\begin{equation*}}
\newcommand{\eeqq}{\end{equation*}}
\newcommand{\beq}{\begin{equation}}
\newcommand{\eeq}{\end{equation}}
\newcommand{\bdm}{\begin{displaymath}}
\newcommand{\edm}{\end{displaymath}}

\theoremstyle{definition}
\newtheorem{theorem}{Theorem}[section]
\newtheorem{definition}{Definition}[section]
\newtheorem{lemma}{Lemma}[section]

\newtheorem{remark}{Remark}[section]

\numberwithin{equation}{section}
\begin{document}
\pagestyle{plain}

\begin{center}
{\Large \bf {A geometric criterion for the existence of chaos based on periodic orbits in continuous-time autonomous systems}}
\\ [0.3in]
XU ZHANG $^a$\ \footnote{Email addresses:\ \ xuzhang08@gmail.com (X. Zhang), eegchen@cityu.edu.hk (G. Chen).\\
2010 AMS subject classifications\ 37C10, 37D45.},\quad
GUANRONG CHEN $^{b}$

\vspace{0.15in}
$^a${\it  Department of Mathematics, Shandong University \\[-0.5ex]
Weihai 264209, Shandong, China}

\vspace{0.15in} $^b${\it Department of Electronic Engineering\\[-0.5ex]
City University of Hong Kong, Hong Kong SAR, China}
\end{center}

\vspace{0.18in}

\baselineskip=20pt

{\bf\large{Abstract.}}\
A new geometric criterion is derived for the existence of chaos in continuous-time autonomous systems in three-dimensional Euclidean spaces, where a type of Smale horseshoe in a subshift of finite type exists, but the intersection of stable and unstable manifolds of two points on a hyperbolic periodic orbit does not imply the existence of a Smale horseshoe of the same type on cross-sections of these two points. This criterion is based on the existence of a hyperbolic periodic orbit, differing from the classical equilibrium-based Shilnikov criterion and the condition of transversal homoclinic or heteroclinic orbits of Poincar\'{e} maps.


\bigskip

\baselineskip=20pt
\section{Introduction}

Consider a continuous-time autonomous system  described by an ordinary differential equation $\dot{x}=\Phi(x),\ x\in\mathbb{R}^3$, where $\Phi:U\to\mathbb{R}^3$ is $C^r$ on some open set $U\subset\mathbb{R}^3$. An equilibrium point $q$ is the state satisfying $\Phi(q) = 0$, that is,  $x = q$ is a solution for all $t$. If the eigenvalues of the Jacobian matrix of the system at the equilibrium have non-zero real parts, namely there is no center manifold, then the equilibrium is called hyperbolic, which can be classified as node, saddle, node-focus, and saddle-focus. Typical hyperbolic chaotic systems include the Lorenz system \cite{Lorenz} and the Chen system \cite{Chen1999} which, with the typical parameter values, have two saddle-foci and one unstable node. Another typical example is the generalized Lorenz system with multi-stability, where two stable equilibria could exist \cite{LeoKuz2015}. For hyperbolic equilibria, there are many well-known criteria on the existence of chaos. In the study of continuous-time autonomous systems, they could be simplified so as to study suitably-defined Poincar\'{e} maps on cross-sections. Since a Poincar\'{e} map represents a discrete dynamical system, many powerful tools could be utilized, such as the Smale horseshoe \cite{Smale1963}, the Smale-Birkhoff Theorem (the existence of a transversal homoclinic orbit) \cite{Robinson}, and the existence of transversal heteroclinic orbits \cite{Bertozzi1988}, to show the existence of chaos in the system. Besides, the Shilnikov criterion \cite{Shilnikov1965, Shilnikov1967, Shilnikov1970,Wiggins2003} and the Melnikov method \cite{Melnikov1963} are useful tools for proving the existence of chaos in a continuous-time autonomous system.

On the other hand, there are some autonomous systems without hyperbolic equilibrium points in three-dimensional spaces, but these systems have chaotic attractors discovered by numerical experiments.
In climate systems, ecosystems, financial markets, engineering applications, mechanical and electromechanical systems, there often exist more than one attractor, which is referred to as multi-stability.  The multi-stability is a typical property of systems without hyperbolic equilibrium points \cite{DudkJafaKapKuznLeonPras2016}. For example,  a mechanical system, discovered by Sommerfeld \cite{Ecke2013, Somm1902}, has oscillations caused by a motor driving an unbalanced weight and resonance capture (Sommerfeld effect), which captures the failure of the rotating system due to the resonant interactions. Other examples include a double-mass mathematical model of the drilling system studied in \cite{MihaVeggWouwNijm2004} and the Rabinovic system describing the interactions of three resonantly coupled plasma waves \cite{PikoRabiTrak1978, Rabi1978}.

There are some other interesting mathematical models without hyperbolic equilibrium points: a chaotic
Chua's circuit \cite{LeoKuz2013},  some rare flows with chaotic attractors but  no equilibrium \cite{Sprott1994},  a chaotic autonomous system with a line of equilibria \cite{MolaJafaSproGolp2013}, a chaotic system with a surface of equilibria \cite{JafariSprottPhamLi2016},  a chaotic system with one and only one stable equilibrium \cite{WangChen2013}, and some others \cite{JafariSprottNazarimehr2015}.  For these  systems with chaotic attractors, their equilibria might be stable, or may not even exist, therefore many classical tools such as the Shilnikov criterion are not applicable to describe their chaotic dynamics. The classical Smale-Birkhoff Theorem, or the existence of a transversal heteroclinic orbit, requires the strong assumption of transversal dynamics, which is difficult to verify in real applications (see Subsection \ref{transsec-1} below for more detailed discussions).

An interesting problem is the mechanism for the existence of chaotic attractors in continuous-time autonomous systems without hyperbolic equilibrium points. In this paper, a geometric criterion is derived to describe the existence of chaos in such systems, revealing the chaos forming mechanism. Specifically, some chaotic dynamics are shown to have  a Smale horseshoe in a subshift of finite type, and the classical intersection mechanism of stable and unstable manifolds of two points on a hyperbolic periodic orbit does not imply the existence of a Smale horseshoe of the same type on cross-sections of these two points (see Remark \ref{trans-asump} and Theorem \ref{mainresult-1} in Section \ref{chaotichidden}).

The rest of the paper is organized as follows. In Section \ref{basic-1}, some basic concepts and useful preliminaries are introduced. In Section \ref{chaotichidden}, the complex dynamics of autonomous systems without hyperbolic equilibrium points are studied. This section is divided into three parts. In the first subsection, the classical transversal homoclinic or heteroclinic orbits are applied to explain the complex dynamics. In the second subsection, a topological model is established for the Smale horseshoe in a subshift of finite type with a particular transition matrix. In the third subsection, a geometric criterion is derived, where some new dynamics are observed with a Smale horseshoe in a subshift of finite type.

\section{Basic Concepts and Preliminaries}\label{basic-1}

First, recall the symbolic dynamics \cite{Robinson}.

Let $m\geq2$ be an integer, $S_0=\{1,2,...,m\}$, and
$$
\textstyle\sum_m:=\{\al=(...,a_{-2},a_{-1},a_0,a_1,a_2,...):\ a_i\in S_0,\ \ i\in\mathbb{Z}\}
$$
be the two-sided sequence space. For any $\al=(...,a_{-1},a_0,a_1,...)$ and $\beta= (...,b_{-1},b_0,b_1,...)\in\sum_m$, the distance between them is
\beqq
d(\al,\beta)=\sum^{\infty}_{i=-\infty}\frac{d(a_i,b_i)}{2^{|i|}},\ d(a_i,b_i)
=\left\{
\begin{array}{ll}
1,& \hbox{if}\  a_i\neq b_i\\
0,& \hbox{if}\ a_i=b_i,
\end{array}
\right.\ i\in\mathbb{Z}.
\eeqq
The shift map $\sigma:\sum_m\to\sum_m$ is defined by $\sigma(\al)=(...,b_{-2},b_{-1},b_0,b_1,b_2...)$, where $\al=(...,a_{-2},a_{-1},a_0,a_1,a_2...)\in\sum_m$ and $b_{i}=a_{i+1}$, $i\in\mathbb{Z}$. The system $\left(\sum_m,\sigma\right)$ is called a two-sided symbolic dynamical system on $m$ symbols, or simply two-sided fullshift on $m$ symbols. A matrix $A=(a_{ij})_{m\times m}$ is called a transition matrix if $a_{ij}=0$ or $1$ for all $2\le i,j\le m$.
For a transition matrix $A$, define
\beqq
\textstyle\sum_{m}(A):=\left\{\be=(...,b_{-2},b_{-1},b_0,b_1,...)\in\textstyle\sum_m:\ a_{b_ib_{i+1}}=1,\ i\in\mathbb{Z}\right\}.
\eeqq
The map $\sigma_A:=\sigma\Big|_{\sum_m(A)}:\sum_m(A)\to\sum_m(A)$ is called the two-sided subshift of finite type with matrix $A$.

\begin{lemma}\cite[Lemma 3.1]{ZhangChen2018}\label{sympositive1}
The topological entropy for the subshift map $\sigma_A:\textstyle\sum_4(A)\to \textstyle\sum_4(A) $ is $\log2>0$, where
\beq\label{shiftmatrix1}
A=\left(
  \begin{array}{cccc}
    0 & 0 & 1 &1\\
    0& 0  & 1 &1\\
    1 & 1& 0 &0\\
    1 & 1& 0& 0
  \end{array}
\right).
\eeq
\end{lemma}

Next, recall the classical Smale horseshoe map.

Consider a square, denoted by $U$, which is a compact subset on a two-dimensional manifold. A horseshoe map $F$ is constructed as follows. The action of the map is defined geometrically by squishing the square along one direction, then stretching the result into a long strip along the perpendicular direction, and finally folding the strip into the shape of a horseshoe, where $F(U)\cap U\neq\emptyset$. This operation is repeated for infinitely many times. An invariant set is formed by $\Ld=\cap_{i\in\mathbb{Z}}F^{i}(U)$, and the dynamics on this invariant set are described by the two-sided fullshift on two symbols \cite{Robinson}.

Now, introduce the Smale horseshoe in a subshift of finite type \cite{Robinson}. For brevity, only a special case is discussed, which will be used in the sequel.

Consider two squares on a two-dimensional manifold, denoted by $U_1$ and $U_2$ respectively, with empty intersection. The horseshoe map $F$ defined on $U_1\cup U_2$ is obtained as follows. The action of the map is defined geometrically by squishing the two squares along the same direction, then stretching the results into two long strips along the perpendicular direction, and finally folding the two strips into the shape of two horseshoes.  $F(U_1)$ and $U_2$ contribute to a horseshoe, and $F(U_2)$ and $U_1$ to another horseshoe. Figure \ref{examplepolyfigure6} illustrates the Smale horseshoe in a subshift of finite type with the matrix $A$ defined in \eqref{shiftmatrix1}, where $U_1$ and $F(U_1)$ are represented by green colors, and $U_2$ and $F(U_2)$ by yellow colors. Note that $F$ is contracting along the horizontal direction and expanding along the vertical direction, where $F(U_1)$ and $U_2$ form a horseshoe, and $F(U_2)$ and $U_1$ form another horseshoe. The set defined by  $\Ld=\cap_{i\in\mathbb{Z}}F^{i}(U_1\cup U_2)$ is invariant under the map, and the dynamics on this invariant set are described by the two-sided subshift of finite type with matrix $A$.

\begin{figure}[h]
\begin{center}
\scalebox{0.8 }{ \includegraphics{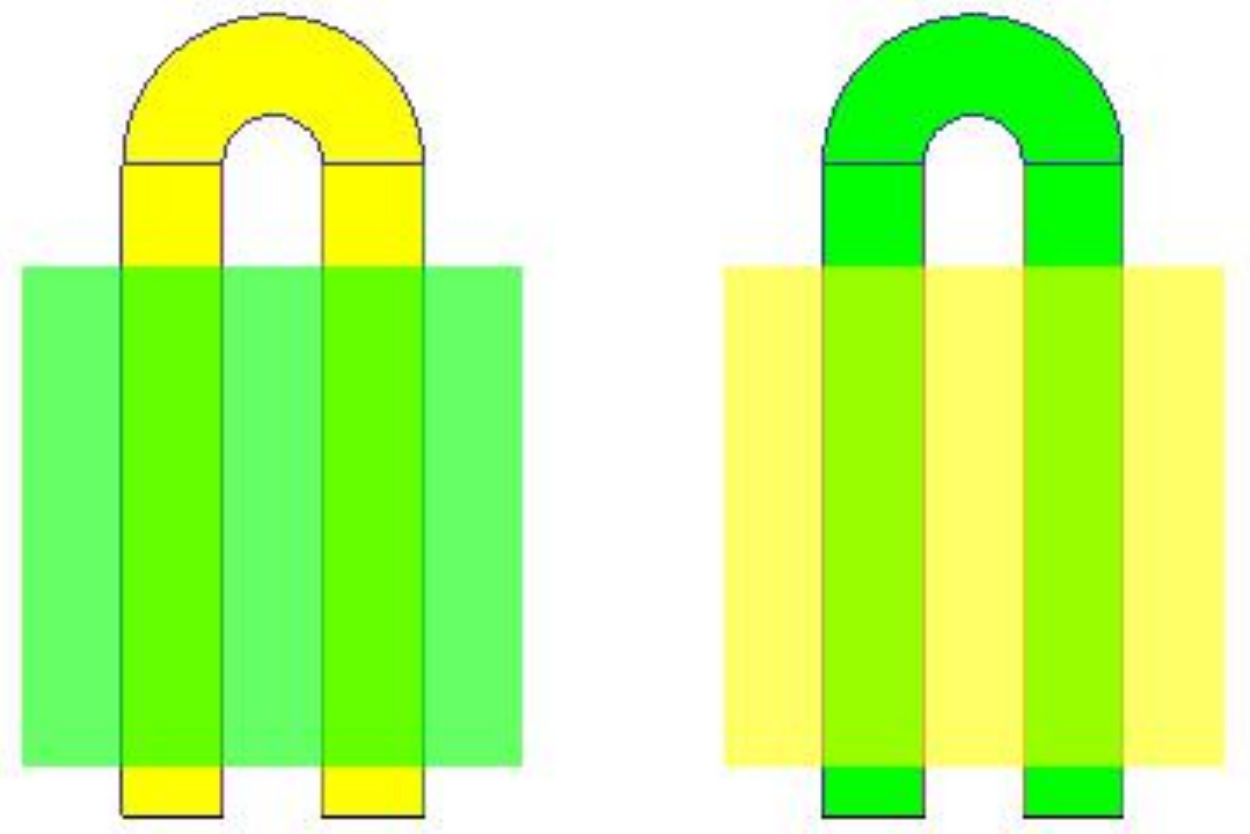}}
\renewcommand{\figure}{Fig.}
\caption{Illustration of the Smale horseshoe in a subshift of finite type with matrix $A$ (Figure 5 in \cite{ZhangChen2018}).
}\label{examplepolyfigure6}
\end{center}
\end{figure}

\section{Chaotic Dynamics of Autonomous Systems without Hyperbolic Equilibria}\label{chaotichidden}

In this section, the complex dynamics of autonomous systems without hyperbolic equilibria are investigated in three parts. For convenience, consider only systems in three-dimensional Euclidean spaces, but higher-dimensional cases and even differential equations defined on smooth manifolds can be similarly discussed.

Consider an ordinary differential equation, $\dot{x}=\Phi(x),\ x\in\mathbb{R}^3$, where $\Phi:U\to\mathbb{R}^3$ is $C^r$ on some open set $U\subset\mathbb{R}^3$. Let $\phi(t,\cdot)$ be a flow generated by this differential equation. For $x_0\in\mathbb{R}^3$, the flow $\phi(t,x_0)$ is the solution to the initial value problem $\dot{x}=\Phi(x)$ with $x(0)=x_0$. Suppose that this equation has a periodic solution of period $T>0$, denoted also by $\phi(t,x_0)$, where $x_0$ is now any point through which this periodic solution passes, namely $\phi(t+T,x_0)=\phi(t,x_0)$. Consider moreover a two-dimensional surface $\sum$ transversal to the vector field at $x_0$, where ``transversal" means that $\Phi(x)\cdot n(x)\neq 0$ with $n(x)$ being the normal to $\sum$ and ``$\cdot$" denoting the vector inner product. The surface $\sum$ is called a cross-section to the vector field.

It is noted that, if $\Phi(x)$ is $C^r$, then $\phi(t,x)$ is $C^r$ (Theorem 7.1.1 in \cite{Wiggins2003}). Thus, there is an open subset $V\subset\sum$ such that the orbits starting in $V$ will return to $\sum$ in a time close to $T$. The associate Poincar\'{e} map is the image of the points in $V$ with their first returns to $\sum$, namely,
\begin{align}\label{poincaremap1}
\nonumber P:&V\to\sum\\
 &x\to\phi(\tau(x),x).
\end{align}
It is clear that $\tau(x_0)=T$ and $P(x_0)=x_0$.

\subsection{Transversal homoclinic/heteroclinic orbits}\label{transsec-1}

In this subsection, the classical homoclinic or heteroclinic orbits are applied to explain the existence of complex dynamics in continuous-time autonomous systems with hidden attractors. Here, the assumption of the existence of transversal homoclinic or heteroclinic orbits is needed.

Consider a periodic orbit of the system and a cross-section of a Poincar\'{e} map containing two points $p$ and $q$ on this periodic orbit. Thus, these two points correspond to a periodic orbit with period two on the Poincar\'{e} map denoted by $P$. Furthermore, if there is a transversal homoclinic orbit corresponding to this periodic orbit, then the following Smale-Birkhoff Theorem could be applied to show the existence of chaos in the system.

\begin{theorem}\cite[Smale-Birkhoff Theorem]{Robinson}
Suppose that $q$ is a transversal homoclinic point corresponding to a hyperbolic periodic point $p$ of a diffeomorphism $f$. For each neighborhood $U$ of $\{p,q\}$, there is a positive integer $n$ such that $f^n$ has a hyperbolic invariant set $\Ld\subset U,$ with $p,q\in\Ld$, on which $f^n$ is topologically conjugate to the two-sided fullshift map on two symbols.
\end{theorem}

Now, assume that there are $m$ periodic orbits in the autonomous system, and there is a cross-section for a Poincar\'{e} map containing one point from each of these $m$ periodic orbits. Clearly, these points are fixed points of the Poincar\'{e} map. If there are transversal heteroclinic orbits with these fixed points, then the following results on the transversal heteroclinic orbits could be applied.

\begin{theorem}\cite[Theorem 2.3.1]{Bertozzi1988}
If a diffeomorphism $f:\mathbb{R}^2\to\mathbb{R}^2$ possesses $m$ fixed points, $p_1,...,p_m$, which are non-degenerate hyperbolic saddle points, and if there exist points $q_i$ at which the unstable manifold $W^u(p_i)$ intersects the stable manifold $W^s(p_{i+1(\mbox{mod}m)})$ transversally for all $i$, then $f$ possesses an invariant set on which some iteration $f^k$ is topologically conjugate to the fullshift on $m$ symbols.
\end{theorem}

\begin{remark}
The classical Shilnikov criterion \cite{Shilnikov1965,Shilnikov1967,Shilnikov1970, Wiggins2003} does not need to consider the Poincar\'{e} map, but it requires the existence of a saddle-focus fixed point for the continuous-time autonomous system, which means that this criterion works only for self-excited systems but not for systems with hidden attractors \cite{LeoKuz2013}.
\end{remark}

\subsection{A topological criterion for the existence of a Smale horseshoe in a subshift of finite type}

In this subsection, a topological criterion is established for the existence of Smale horseshoe in a subshift of finite type with matrix $A$, which is a transition matrix introduced in \eqref{shiftmatrix1}. A similar criterion could be derived for other transition matrices. This is a direct extension of the classical Conley-Moser condition \cite{Moser1973, Wiggins2003}.

\begin{definition}\cite[Definition 25.1.1]{Wiggins2003}
Consider a region $[a,b]\times [c,d]\subset\mathbb{R}^2$, where $b-a=1$ and $d-c=1$. A $\mu_v$-vertical curve is the graph of a function $v(y)$ that satisfies
$$
a\leq v(y)\leq b,\ |v(y_1)-v(y_2)|\leq \mu_{v}|y_1-y_2|\ \mbox{for}\ c\leq y_1,y_2\leq d.
$$
Similarly, a $\mu_{h}$-horizontal curve is the graph of a function $h(x)$ that satisfies
$$
c\leq h(x)\leq d,\ |h(x_1)-h(x_2)|\leq\mu_h|x_1-x_2|\ \mbox{for}\ a\leq x_1,x_2\leq b.
$$
\end{definition}

\begin{definition}\cite[Definition 25.1.2]{Wiggins2003}
Given two non-intersecting $\mu_v$-vertical curves, $v_1(y)<v_2(y)$ and $y\in[c,d]$, define a $\mu_v$-vertical strip by
$$
V=\{(x,y)\in[a,b]\times[c,d]\subset\mathbb{R}^2:\ x\in[v_1(y),v_2(y)],\ y\in[c,d]\}.
$$
Similarly, given two non-intersecting $\mu_h$-horizontal curves, $h_1(x)<h_2(x)$ and $x\in[a,b]$, define a $\mu_h$-horizontal strip by
$$
H=\{(x,y)\in[a,b]\times[c,d]\subset\mathbb{R}^2:\ y\in[h_1(x),h_2(x)],\ x\in[a,b]\}.
$$
The widths of the horizontal and vertical strips are defined respectively as
$$
d(H)=\max_{x\in[a,b]}|h_2(x)-h_1(x)|,
$$
$$
d(V)=\max_{y\in[c,d]}|v_2(y)-v_1(y)|.
$$
\end{definition}

\begin{lemma}\cite[Lemma 25.1.3]{Wiggins2003}
\begin{itemize}
\item[(i)] If $V^1\supset V^2\supset\cdots\supset V^k\supset\cdots$ is a nested sequence of $\mu_v$-vertical strips, with $d(V^k)\to 0$ as $k\to\infty$, then $V^{\infty}:=\cap^{\infty}_{k=1}V^k$ is a $\mu_v$-vertical curve.
\item[(ii)] If $H^1\supset H^2\supset\cdots\supset H^k\supset\cdots$ is a nested sequence of $\mu_h$-horizontal strips, with $d(H^k)\to 0$ as $k\to\infty$, then $H^{\infty}:=\cap^{\infty}_{k=1}H^k$ is a $\mu_h$-horizontal curve.
\end{itemize}
\end{lemma}

\begin{lemma}\cite[Lemma 25.1.4]{Wiggins2003}
Suppose that $0\leq\mu_v\mu_h<1$. Then, a $\mu_v$-vertical curve and a $\mu_h$-horizontal curve intersect at a unique point.
\end{lemma}

Now, consider a map $f:D_1\cup D_2\to\mathbb{R}^2$, where
$$
D_1=\{(x,y)\in\mathbb{R}^2:\ -2\leq x\leq -1,\ 0\leq y\leq1\},
$$
$$
D_2=\{(x,y)\in\mathbb{R}^2:\ 1\leq x\leq 2,\ 0\leq y\leq1\}.
$$
Consider, also, a finite set $S=\{1,2,3,4\}$, four $\mu_h$-horizontal strips, $H^{j}_{i}\subset D_j$, $1\leq i,j\leq2$, and four $\mu_v$-vertical strips, $V^{j}_{i}\subset D_j$, $1\leq i,j\leq2$.

Suppose that $f$ maps $H^{1}_{i}$ homeomorphically onto $V^2_{i}$, and maps $H^{2}_{i}$ homeomorphically onto $V^1_{i}$, $1\leq i\leq 2$. Suppose, moreover, that $f$ satisfies the following two assumptions.

{\bf Assumption 1}: With $0\leq\mu_v\mu_h<1$, the horizontal boundaries of $H^1_i$ are mapped to the horizontal boundaries of $V^2_i$ and the vertical boundaries of $H^1_i$ are mapped to the vertical boundaries of $V^2_i$; the horizontal boundaries of $H^2_i$ are mapped to the horizontal boundaries of $V^1_i$ and the vertical boundaries of $H^2_i$ are mapped to the vertical boundaries of $V^1_i$.

{\bf Assumption 2}: Suppose that $H$ is a $\mu_h$-horizontal strip contained in $H^2_1\cup H^2_2\subset D_2$, and that
$$
f^{-1}(H)\cap H^1_i=\widetilde{H}^1_i,\ 1\leq i\leq 2,
$$
is a $\mu_h$-horizontal strip. Moreover,
$$
d(\widetilde{H}^1_i)\leq\nu_h d(H)\ \mbox{for some}\ 0<\nu_h<1.
$$
Similarly, suppose that $V$ is a $\mu_v$-vertical strip contained in $V^1_1\cup V^1_2\subset D_1$. Then,
$$
f(V)\cap V^2_i=\widetilde{V}^2_i,\ 1\leq i\leq 2,
$$
is a $\mu_v$-vertical strip. Moreover,
$$
d(\widetilde{V}^2_i)\leq\nu_v d(V)\ \mbox{for some}\ 0<\nu_v<1.
$$

Similar assumptions apply to $H$, which is a $\mu_h$-horizontal strip contained in $H^1_1\cup H^1_2\subset D_1$, and  $V$ is a $\mu_v$-vertical strip, which is contained in $V^2_1\cup V^2_2\subset D_2$.

An illustrative diagram is given in Figure \ref{examplefig-6}, where $H^1_1$ and $H^1_2$ are in green color in $D_1$, and $V^1_1$ and $V^1_2$ are in yellow color in $D_1$; $H^2_1$ and $H^2_2$ are in yellow color in $D_2$, and $V^2_1$ and $V^2_2$ are in green color in $D_2$.

\begin{figure}[h]
\begin{center}
\scalebox{0.8 }{ \includegraphics{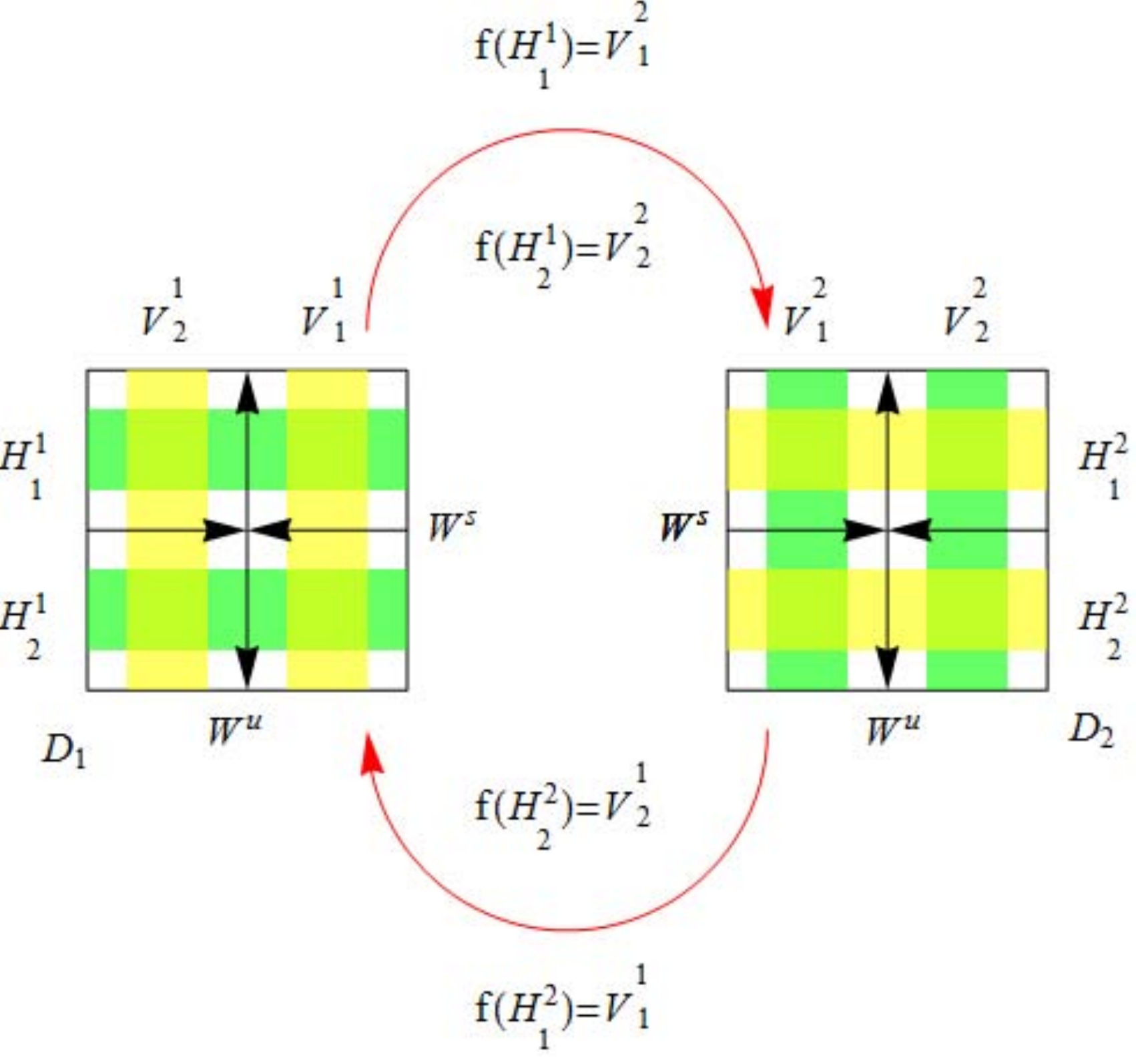}}
\renewcommand{\figure}{Fig.}
\caption{Illustrative diagram for the horseshoe in a subshift of finite type with matrix $A$, where $H^1_1$ and $H^1_2$ are in green color in $D_1$, and $V^1_1$ and $V^1_2$ are in yellow color in $D_1$; $H^2_1$ and $H^2_2$ are in yellow color in $D_2$, and $V^2_1$ and $V^2_2$ are in green color in $D_2$.
}\label{examplefig-6}
\end{center}
\end{figure}

Now, the following result can be established.

\begin{theorem}\label{chaoshorsesho-1}
Suppose that $f$ satisfies Assumptions 1 and 2. Then, $f$ has an invariant Cantor set $\Ld$, on which $f$ is topologically conjugate to a subshift of finite type with the matrix $A$ specified in \eqref{shiftmatrix1}, such that the following relations hold:
\begin{center}
\begin{tikzpicture}
    % set up the nodes
    \node (E) at (0,0) {$\Ld$ };
    \node[right=of E] (F) at (4,0){$\Ld $};
    \node[below=of F] (A) {$\sum_4(A)$};
    \node[below=of E] (Asubt) {$\sum_4(A)$};
    % draw arrows and text between them
   \draw[->] (E)--(F) node [midway,above] {$f$};
    \draw[->] (F)--(A) node [midway,right] {$\psi$}
                node [midway,left] {};
    \draw[->] (Asubt)--(A) node [midway,below] {$\si$}
                node [midway,above] {};
    \draw[->] (E)--(Asubt) node [midway,left] {$\psi$};
    %\draw[dashed] (Asubt)--(F);
\end{tikzpicture}
\end{center}
where $\psi$ is a homeomorphism mapping $\Ld$ onto $\sum_4(A)$.
\end{theorem}

\begin{proof}
It follows from arguments similar to the proof of \cite[Theorem 25.1.5]{Wiggins2003}.
\end{proof}

\subsection{A geometric criterion for the existence of chaos}

Consider an ordinary differential equation, $\dot{x}=\Phi(x),\ x\in\mathbb{R}^3$, where $\Phi:\mathbb{R}^3\to\mathbb{R}^3$ is differentiable. For any initial point $x_0\in\mathbb{R}^3$, let the solution to the corresponding initial value problem be $\phi(t,x_0)$, called a flow and denoted by $\phi^t$ for simplicity.

\begin{definition}\cite{Robinson}
An invariant set $\Ld$ for the flow $\phi^t$ defined on a smooth manifold $M$ has a hyperbolic structure, namely $\Ld$ is a hyperbolic invariant set, provided that
\begin{itemize}
\item[(i)] at each point $p$ in $\Ld$, the tangent space to $M$ can be split as the direct sum of $\mathbb{E}^u_p$, $\mathbb{E}^s_p$, and $\mbox{span}(\Phi(p))$:
    \beqq
    T_p(M)=\mathbb{E}^u_p\oplus\mathbb{E}^s_p\oplus\mbox{span}(\Phi(p));
    \eeqq
\item[(ii)] the above splitting is invariant under the action of the derivative, in the sense that
\beqq
D(\phi^t)_p\mathbb{E}^u_p=\mathbb{E}^u_{\phi^t(p)},\ D(\phi^t)_p\mathbb{E}^s_p=\mathbb{E}^s_{\phi^t(p)},\ D(\phi^t)_p\Phi(p)=\Phi(\phi^t(p));
\eeqq
\item[(iii)] $\mathbb{E}^u_p$ and $\mathbb{E}^s_p$ vary continuously with $p$;
    \item[(iv)] there exist $\mu>0$ and $C\geq1$ such that, for any $t\geq0$,
        \beqq
        |D\phi^t_p v^s|\leq Ce^{-\mu t}|v^s|\ \ \mbox{for}\ v^s\in\mathbb{E}^s_p,
        \eeqq
        \beqq
        |D\phi^{-t}_p v^u|\leq Ce^{-\mu t}|v^u|\ \ \mbox{for}\ v^u\in\mathbb{E}^u_p.
        \eeqq
\end{itemize}
\end{definition}

Furthermore, assume that there is a hyperbolic periodic orbit with period $T>0$; that is, for any point on the periodic orbit, there exists a cross-section passing through this point. The point on the periodic orbit is a saddle fixed point of the Poincar\'{e} map. For an illustration of the hyperbolic periodic orbit by using Poincar\'{e} map, see Figures 10.1.2 and 10.1.3 in \cite{Wiggins2003}.

Now, take any point $x_0$ on this periodic orbit. Its stable and unstable manifolds are defined as follows:
\beqq
W^s(x_0)=\bigg\{x\in\mathbb{R}^3:\ \lim_{t\to+\infty}d(\phi(t,x),\phi(t,x_0))=0\bigg\}
\eeqq
and
\beqq
W^u(x_0)=\bigg\{x\in\mathbb{R}^3:\ \lim_{t\to-\infty}d(\phi(t,x),\phi(t,x_0))=0\bigg\},
\eeqq
where $d(\cdot,\cdot)$ is a metric induced by the Euclidean norm on $\mathbb{R}^3$. Similarly, its local stable and unstable manifolds are defined by
\beqq
W^s_r(x_0)=\bigg\{x\in\mathbb{R}^3:\ d(x,x_0)<r\ \mbox{and}\  \lim_{t\to+\infty}d(\phi(t,x),\phi(t,x_0))=0\bigg\}
\eeqq
and
\beqq
W^u_r(x_0)=\bigg\{x\in\mathbb{R}^3:\ d(x,x_0)<r\ \mbox{and}\  \lim_{t\to-\infty}d(\phi(t,x),\phi(t,x_0))=0\bigg\},
\eeqq
where $r>0$ is a positive constant. For convenience, $W^s_{loc}(x_0)$ and $W^u_{loc}(x_0)$ represent the local stable and unstable manifolds, both with sufficiently small radii.

Actually, the stable and unstable manifolds can be defined by another method. Consider a discrete map induced by the flow $g=\phi(T,\cdot):\mathbb{R}^3\to\mathbb{R}^3$. It is evident that, for any point $x_0$ on this periodic orbit, one has $\phi(T,x_0)=x_0$; that is, any point on this periodic orbit is a fixed point. By the assumption of the hyperbolic periodic orbit, for the discrete map $g$ and any point $x_0$ on this periodic orbit, there exists a decomposition of the tangent space, $\mathbb{E}^s_{x_0}\oplus \mathbb{E}^u_{x_0}\oplus \mathbb{E}^c_{x_0}$, where $\mathbb{E}^s_{x_0}$ is the contraction direction, $\mathbb{E}^u_{x_0}$ is the expansion direction, and $\mathbb{E}^c_{x_0}=\mbox{span}\Phi(x_0)$ is the center direction (flow direction), and the derivative of the $g$ along this center direction $\mathbb{E}^c_{x_0}$ is $1$. The local stable and unstable manifolds with respect to $g$ are denoted by $W^s_{loc}(x_0,g)$ and $W^u_{loc}(x_0,g)$, respectively, for which the existence of the local stable and unstable manifolds are guaranteed by the Stable Manifold Theorem \cite{Robinson,Wiggins2003}. The stable and unstable manifolds of $x_0$ with respect to the flow $\phi^t$ can be expressed by
\beqq
W^s(x_0)=\bigcup_{k\in\mathbb{Z},k\leq0}\phi(kT,W^s_{loc}(x_0,g)),
\eeqq
\beqq
W^u(x_0)=\bigcup_{k\in\mathbb{Z},k\geq0}\phi(kT,W^u_{loc}(x_0,g)).
\eeqq

\begin{remark}\label{poincarestable1}
For any point $x_0$ on the periodic orbit and any cross-section containing the point $x_0$, suppose that $x_0$ is a saddle fixed point of the Poincar\'{e} map, denoted by $P$. Then, there exist local stable and unstable manifolds of $x_0$ with respect to the Poincar\'{e} map, denoted by $W^s_{loc}(x_0,P)$ and $W^u_{loc}(x_0,P)$, respectively. However, $W^s_{loc}(x_0,P)$ and $W^u_{loc}(x_0,P)$ cannot be used to define the stable and unstable manifolds of $x_0$ with respect to the flow $\phi^t$, because in the present case, in the Poincar\'{e} map defined by \eqref{poincaremap1}, $\tau(x)$ might not be equal to $T$. Yet, to define the stable and unstable manifolds of $x_0$ with respect to the flow $\phi^t$, it is required that $\tau(x)=T$; that is, the return time should be equal for all points on the same cross-section.
\end{remark}

Now, the following main result is obtained.

\begin{theorem}\label{mainresult-1}
Consider an ordinary differential equation, $\dot{x}=\Phi(x)$, $x\in\mathbb{R}^3$, where $\Phi:\mathbb{R}^3\to\mathbb{R}^3$ is differentiable.
Assume that
\begin{itemize}
\item there exist a hyperbolic periodic orbit of period $T>0$, and two points $p$ and $q$ on this periodic orbit with $\phi(\tfrac{T}{2},p)=q$ and $\phi(\tfrac{T}{2},q)=p$;
\item there are an open subset $\Upsilon_1\subset W^u_{loc}(p)\times W^s_{loc}(p)$ containing a line segment of $W^u_{loc}(p)$,  an open subset $\Upsilon_2\subset W^u_{loc}(q)\times W^s_{loc}(q)$ containing a line segment of $W^u_{loc}(q)$, and two positive integers $m_p$ and $m_q$ such that $\phi(\tfrac{T}{2}+m_pT,\Upsilon_1)$ contains a segment of $W^s_{loc}(q)$ with $\phi(\tfrac{T}{2}+m_pT,\Upsilon_1)\subset W^u_{loc}(q)\times W^s_{loc}(q)$, and $\phi(\tfrac{T}{2}+m_qT,\Upsilon_2)$ contains a segment of $W^s_{loc}(p)$ with $\phi(\tfrac{T}{2}+m_qT,\Upsilon_2)\subset W^u_{loc}(p)\times W^s_{loc}(p)$.
\end{itemize}

Then, there exist a positive integer $m$ and an invariant set $\Ld\subset\mathbb{R}^3$ such that the following relations hold:
\begin{center}
\begin{tikzpicture}
    % set up the nodes
    \node (E) at (0,0) {$\Ld$ };
    \node[right=of E] (F) at (4,0){$\Ld $};
    \node[below=of F] (A) {$\sum_4(A)$};
    \node[below=of E] (Asubt) {$\sum_4(A)$};
    % draw arrows and text between them
   \draw[->] (E)--(F) node [midway,above] {$\phi(\tfrac{T}{2}+mT,\cdot)$};
    \draw[->] (F)--(A) node [midway,right] {$\Psi$}
                node [midway,left] {};
    \draw[->] (Asubt)--(A) node [midway,below] {$\si$}
                node [midway,above] {};
    \draw[->] (E)--(Asubt) node [midway,left] {$\Psi$};
    %\draw[dashed] (Asubt)--(F);
\end{tikzpicture}
\end{center}
where $\Psi$ is a homeomorphism from $\Ld$ to $\sum_4(A)$, which is a topological conjugacy.
\end{theorem}

\begin{remark}\label{trans-asump}
 The conventional assumption that $W^u(p)\cap W^s(q)\neq\emptyset$ and $W^s(p)\cap W^u(q)\neq\emptyset$ does not imply the existence of Smale horseshoe in this situation.

Suppose that there is $p^*\in W^u(p)\cap W^s(q)$. Since $p$ and $q$ are on the periodic orbit of period $T$, one has $\lim_{k\to-\infty}d(\phi(t,p),\phi(t,p^*))=0$. This, together with $\phi(kT,p)=p$, implies that $\lim_{k\to-\infty}\phi(kT,p^*)=p$. Similarly, $\lim_{k\to+\infty}\phi(kT,p^*)=q$.
Therefore, this horseshoe structure could not be obtained by the conventional assumptions. On the other hand, the conventional assumptions might imply the existence of the classical Smale horseshoes.
\end{remark}

\begin{remark}
A simplified model for a Smale horseshoe is illustrated by Figure \ref{examplepolyfigure6}. For a particular discrete system with a Smale horseshoe in a subshift of finite type with matrix $A$, see \cite{ZhangChen2018}.
\end{remark}

\begin{remark}
Continuous-time autonomous dynamical systems can be classified into two types according to their attractors: self-excited systems and hidden-attraction systems. For a system, if its basin of attraction intersects arbitrarily small neighborhoods of an existing equilibrium, it is called self-excited; otherwise, namely if its basin of attraction does not intersect a small neighborhood of an existing equilibrium, it is called hidden \cite{LeonKuznVaga2011}. For example, the chaotic Lorenz system \cite{Lorenz} and Chen system \cite{Chen1999} with the typical parameter values are self-excited systems, and there are some other systems with hidden attractors \cite{DudkJafaKapKuznLeonPras2016, JafariSprottNazarimehr2015}.

Maps with hidden dynamics could be similarly defined. Some hidden attractors in one-dimensional maps were obtained in \cite{JafariPhamMoghtadaeiKingni2016} by extending the Logistic map. A class of two-dimensional quadratic maps with hidden dynamics were studied in \cite{JiangLiuWeiZhang2016}. A one-dimensional and a two-dimensional generalized H\'{e}non map with hidden dynamics were studied in \cite{ZhangChen2018}.

Theorem \ref{mainresult-1} above can be used to study the chaotic dynamics of autonomous systems with hidden chaotic attractors.
\end{remark}

The following detailed analysis of chaotic dynamics is divided into two steps:
\begin{itemize}
\item[1.] illustrating the vector field on a neighborhood of the hyperbolic periodic orbit;
\item[2.] analyzing the existence of a Smale horseshoe in a subshift of finite type with matrix $A$.
\end{itemize}

{\bf Step 1.} Illustrating the vector field on a neighborhood of the hyperbolic periodic orbit (not an equilibrium).

First, two examples with periodic orbits are provided to illustrate a system with a saddle-focus periodic orbit.

{\bf{Example 1.}}
Consider the following local representation near a periodic orbit:
$$
\left\{
  \begin{array}{ll}
    \dot{x}&=-y+(x^2+y^2-1)F_1(x,y,z):=G_1(x,y,z), \\
    \dot{y}&=x+(x^2+y^2-1)F_2(x,y,z):=G_2(x,y,z), \\
    \dot{z}&=(z+x^2+y^2-1)F_3(x,y,z):=G_3(x,y,z).
  \end{array}
\right.
$$

It is evident that $\{(x,y,z)\in\mathbb{R}^3:\ x^2+y^2=1,\ z=0\}$ is a periodic solution to this system. Next, it is shown that under certain conditions it has a saddle-focus near this periodic orbit.

By direct calculation, one has
\beqq
\frac{\partial G_1}{\partial x}=2xF_1+(x^2+y^2-1)\frac{\partial F_1}{\partial x},\ \frac{\partial G_1}{\partial y}=-1+2yF_1+(x^2+y^2-1)\frac{\partial F_1}{\partial y},\ \frac{\partial G_1}{\partial z}=(x^2+y^2-1)\frac{\partial F_1}{\partial z},
\eeqq
\beqq
\frac{\partial G_2}{\partial x}=1+2xF_2+(x^2+y^2-1)\frac{\partial F_2}{\partial x},\ \frac{\partial G_2}{\partial y}=2yF_2+(x^2+y^2-1)\frac{\partial F_2}{\partial y},\ \frac{\partial G_2}{\partial z}=(x^2+y^2-1)\frac{\partial F_2}{\partial z},
\eeqq
\beqq
\frac{\partial G_3}{\partial x}=2xF_3+(z+x^2+y^2-1)\frac{\partial F_3}{\partial x},\ \frac{\partial G_3}{\partial y}=2yF_3+(z+x^2+y^2-1)\frac{\partial F_3}{\partial y},\
\eeqq
\beqq
\frac{\partial G_3}{\partial z}=F_3+(z+x^2+y^2-1)\frac{\partial F_3}{\partial z}.
\eeqq
So, the Jacobian on this periodic orbit is
$$
\left(
  \begin{array}{ccc}
    \tfrac{\partial G_1}{\partial x} &  \tfrac{\partial G_1}{\partial y} &  \tfrac{\partial G_1}{\partial z} \\
   \tfrac{\partial G_2}{\partial x} &  \tfrac{\partial G_2}{\partial y} &  \tfrac{\partial G_2}{\partial z} \\
    \tfrac{\partial G_3}{\partial x} &  \tfrac{\partial G_3}{\partial y} &  \tfrac{\partial G_3}{\partial z} \\
  \end{array}
\right)=\left(
  \begin{array}{ccc}
    2xF_1 & -1+2yF_1 & 0 \\
    1+2xF_2 & 2yF_2 & 0 \\
    2xF_3 & 2yF_3 & F_3 \\
  \end{array}
\right).
$$
So, the eigenvalues are the solutions to the equation
\begin{align*}
&\left|
  \begin{array}{ccc}
    2xF_1-\ld & -1+2yF_1 & 0 \\
    1+2xF_2 & 2yF_2-\ld & 0 \\
    2xF_3 & 2yF_3 & F_3-\ld \\
  \end{array}
\right|=
\left|
    \begin{array}{cc}
      2xF_1-\ld & -1+2yF_1 \\
     1+2xF_2 & 2yF_2-\ld \\
    \end{array}
  \right|(F_3-\ld)\\
=&(\ld^2+(-2yF_2-2xF_1)\ld+(1+2xF_2-2yF_1))(F_3-\ld)=0.
\end{align*}
It is obvious that one eigenvalue is $F_3$, and the other two eigenvalues are the solutions to the quadratic polynomial $\ld^2+(-2yF_2-2xF_1)\ld+(1+2xF_2-2yF_1)=0$. If the following inequalities hold (along the periodic orbit):
\begin{itemize}
\item $F_3$ is real and $F_3>0$,
\item $-2yF_2-2xF_1>0$,
\item $(-2yF_2-2xF_1)^2-4(1+2xF_2-2yF_1)<0$,
\end{itemize}
then this periodic orbit is called saddle-focus.

For example, choose $F_3=\ga>0$, $F_1=-\al x(x^2+y^2)$, and $F_2=-\al y(x^2+y^2)$, where $\ga$ and $\al\in(0,1)$ are two constants. By direct calculation, we have $-2yF_2-2xF_1=2\al>0$, $(-2yF_2-2xF_1)^2-4(1+2xF_2-2yF_1)=4\al^2-4=4(\al^2-1)<0$, the eigenvalues are solutions to the  equation $\ld^2+2\al\ld+1=0$, the solutions are $\frac{-2\al\pm\sqrt{4(\al^2-1)}}{2}=-\al\pm\sqrt{1-\al^2}i$.

{\bf{Example 2.}}
Consider another example with non-constant eigenvalues:
$$
\left\{
  \begin{array}{ll}
    \dot{x}&=-y+(z+x^2+y^2-1)F_1(x,y,z), \\
    \dot{y}&=x+(z+x^2+y^2-1)F_2(x,y,z), \\
    \dot{z}&=(z+x^2+y^2-1)F_3(x,y,z).
  \end{array}
\right.
$$
It is evident that $\{(x,y,z)\in\mathbb{R}^3:\ x^2+y^2=1,\ z=0\}$ is a periodic solution to this system. Next, it is shown that under certain conditions it has a saddle-focus near this periodic orbit.

The Jacobian on this periodic orbit is
$$
\left(
  \begin{array}{ccc}
    2xF_1 & -1+2yF_1 & F_1 \\
    1+2xF_2 & 2yF_2 & F_2 \\
    2xF_3 & 2yF_3 & F_3 \\
  \end{array}
\right)
$$

The characteristic function is
$$
\left|
  \begin{array}{ccc}
    2xF_1-\ld & -1+2yF_1 & F_1 \\
    1+2xF_2 & 2yF_2-\ld & F_2 \\
    2xF_3 & 2yF_3 & F_3-\ld \\
  \end{array}
\right|=-\ld^3+\ld^2(2xF_1+2yF_2+F_3)+\ld(1+2xF_2-2yF_1)-F_3=0.
$$
Let $F_1=-\al x(x^2+y^2)$ and $F_2=-\al y(x^2+y^2)$, where $\al>0$ is a constant. The eigenvalue function is simplified as follows $\ld^3+\ld^2(2\al-F_3)-\ld+F_3=0$. For this cubic equation, if it has a positive solution, and two conjugate complex solutions, then the corresponding periodic orbit is saddle-focus.

Suppose $\ld=t-\frac{2\al-F_3}{3}$. Then, one has
\begin{align*}
 \ld^3+(2\al-F_3)\ld^2-\ld+F_3
=t^3+\bigg(-\frac{(2\al-F_3)^2}{3}-1\bigg)t+\frac{2}{27}(2\al-F_3)^3+\frac{2\al-F_3}{3}+F_3.
\end{align*}
For the cubic polynomial, if
$$
\bigg(\frac{1}{27}(2\al-F_3)^3+\frac{2\al-F_3}{6}+\frac{F_3}{2}\bigg)^2+\bigg(-\frac{(2\al-F_3)^2}{9}-\frac{1}{3}\bigg)^3>0,
$$
then there are a positive solution, and two conjugate complex solutions.

It is evident that if $F_3\gg\al>0$  is sufficiently large restricted to the periodic orbit $\{(x,y,z)\in\mathbb{R}^3:\ x^2+y^2=1,\ z=0\}$, then the expanding direction is almost parallel to the $z$-axis.

Suppose that $\ld_0$ is the eigenvalue in the unstable subspace. Hence, the direction of the unstable subspace is parallel to the eigenvector, which is the solution to the following equation:
$$
\left(
  \begin{array}{ccc}
    2xF_1-\ld_0 & -1+2yF_1 & F_1 \\
    1+2xF_2 & 2yF_2-\ld_0 & F_2 \\
    2xF_3 & 2yF_3 & F_3-\ld_0 \\
  \end{array}
\right)
\left(
  \begin{array}{c}
    x_1 \\
    x_2 \\
    x_3 \\
  \end{array}
\right)=
\left(
  \begin{array}{c}
    0 \\
    0 \\
    0 \\
  \end{array}
\right).
$$
Since this $\ld_0$ is a single root, and $F_3$ is sufficiently large, one has $\left|
  \begin{array}{cc}
    2xF_1-\ld_0 & -1+2yF_1 \\
    1+2xF_2 & 2yF_2-\ld_0 \\
  \end{array}
\right|\neq0$. Hence, the eigenvector can be chosen as a solution to the following equation:
$$
\left(
  \begin{array}{cc}
    2xF_1-\ld_0 & -1+2yF_1  \\
    1+2xF_2 & 2yF_2-\ld_0  \\
  \end{array}
\right)
\left(
  \begin{array}{c}
    x_1 \\
    x_2 \\
  \end{array}
\right)=
\left(
  \begin{array}{c}
    -F_1 \\
    -F_2 \\
  \end{array}
\right).
$$
Hence, the eigenvector is $(\tfrac{F_1\ld_0-F_2}{\ld^2_0+2\al\ld_0+1},\tfrac{F_1+\ld_0 F_2}{\ld_0^2+2\al\ld_0+1},1)$. For sufficiently large $\ld_0$, this vector is parallel to the $z$-axis.

Next, a geometric description of the general vector field near the periodic orbit is provided.

Let $\ga=\{\phi(t,x_0):\ t\in[0,T)\}$ be the periodic orbit, where $x_0$ is any point on this periodic orbit. Consider a cross-section passing this point $x_0$ and the corresponding Poincar\'{e} map $P$. The stable and unstable manifolds of the periodic orbit are
\beq\label{periodicmanifold}
W^s(\ga)=\bigcup_{t\leq 0}\phi(t,W^s_{loc}(x_0,P))\ \ \mbox{and}\ \ W^u(\ga)=\bigcup_{t\geq0}\phi(t,W^u_{loc}(x_0,P)),
\eeq
where $W^s_{loc}(x_0,P)$ and $W^u_{loc}(x_0,P)$ are the local stable and unstable manifolds at the point $x_0$ with respect to the Poincar\'{e} map. In \eqref{periodicmanifold}, $W^s(\ga)$ and $W^u(\ga)$ are two-dimensional surfaces, which intersect on the closed curve $\ga$. For an illustrative diagram, see Figure 10.1.3 in \cite{Wiggins2003}.

Note that a periodic orbit might be homeomorphic to a knot \cite{BirmanWilliams1983, Williams1984}. For illustration, consider a periodic orbit that is homeomorphic to a circle. Even for this simple case, the vector field near the periodic orbit might not be simple, since there might exist M\"{o}bius bands (or M\"{o}bius strips) contained in $W^s(\ga)$ and $W^u(\ga)$, respectively. For example, construct a simple vector field defined in a neighborhood of a periodic orbit as follows: start from a region $V=[-1,1]\times[-1,1]\times[1,2]\subset\mathbb{R}^3$ with the vector field described by
\beqq
\vec{r}=-x\frac{\partial}{\partial x}+y \frac{\partial}{\partial y}+\frac{\partial}{\partial z},\ w=(x,y,z)\in V.
\eeqq
For this set $V$, consider the quotient space given by the following equivalent relationship:
\beqq
\widetilde{V}=\{(x,y,z)\in V:\ (x,y,1)\sim(-x,-y,2)\},
\eeqq
where $(x,y,1)$ and $(-x,-y,2)$ are regarded as the same point in the quotient space $\widetilde{V}$. By the definition of the vector field $\vec{r}$, this induces a natural continuous vector field on $\widetilde{V}$, denoted by $\vec{R}$. For convenience, use the coordinates on $V$ to represent the point on $\widetilde{V}$. For the vector field $\vec{R}$, it is evident that there exists a periodic orbit, $\{(0,0,t):\ t\in[0,1]\}$. This periodic orbit is hyperbolic, which is contracting along the $x$-direction and expanding along the $y$-direction. There exist M\"{o}bius bands (or M\"{o}bius strips) contained in $W^s(\ga)$ and $W^u(\ga)$, respectively.
For more discussions on the existence of M\"{o}bius bands for vector fields and their corresponding dynamics, see \cite[Section 27.2]{Wiggins2003}.

{\bf Step 2.} Analyzing the chaotic dynamics. The discussions are divided into four parts.

{\bf (i)} Consider the local coordinates on the cross-sections at the points $p$ and $q$.

Since the periodic orbit is hyperbolic, there are local stable and unstable manifolds for both the points $p$ and $q$. Take two cross-sections, $\Pi_p$ and $\Pi_q$ at the points $p$ and $q$, respectively, where these two cross-sections are generated by the product of the local stable and unstable manifolds. And, the local coordinates on the cross-sections are taken as $W^s_{loc}(w)$ and $W^u_{loc}(w)$, respectively, denoted by $W^s(w)$ and $W^u(w)$ for simplicity.

Figure \ref{examplefigure1} shows an illustrative diagram for the hypotheses of the system. Consider a periodic orbit (in blue color) with period $T>0$, and two points $p$ and $q$ on the periodic orbit satisfying $\phi(\tfrac{T}{2},p)=q$ and $\phi(\tfrac{T}{2},q)=p$. By the assumption of $\Upsilon_1$, the red line represents the orbit from $W^u(q)$ to $W^s(p)$, and the green line refers to the orbit from $W^u(p)$ to $W^s(q)$. Note that $W^{\si}_{loc}(p)\setminus\{p\}$ and $W^{\si}_{loc}(q)\setminus\{q\}$ contain two disjoint curves, respectively. As seen from Remark \ref{poincarestable1}, for any cross-section of the point $w$ on the periodic orbit, the local stable and unstable manifolds of $w$ might not be contained in the cross-section.

\begin{figure}[hpt]
\begin{center}
\scalebox{0.6 }{ \includegraphics{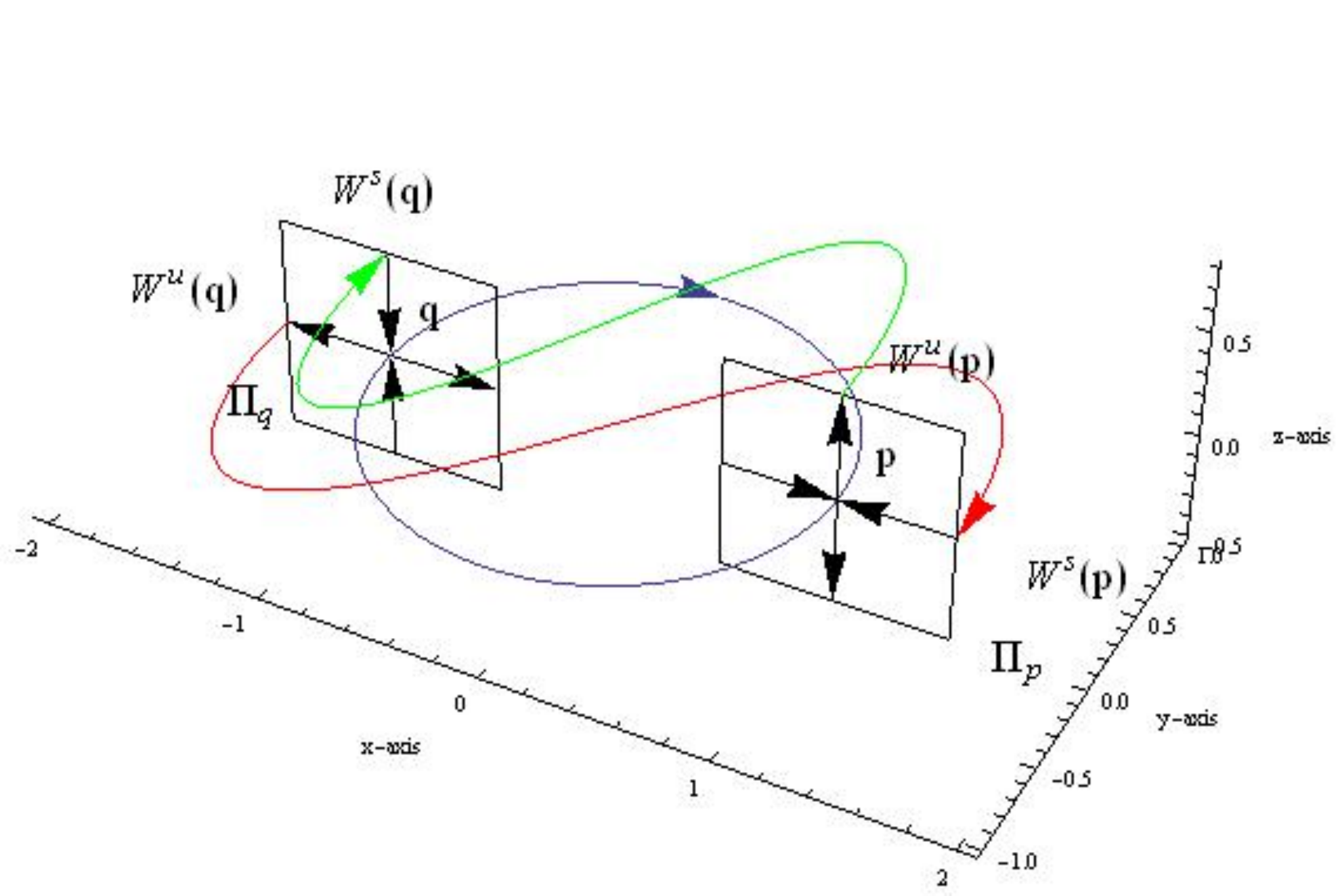}}
\renewcommand{\figure}{Fig.}
\caption{Illustrative diagram for the hypotheses of the system, where the periodic orbit with period $T$ is in blue color, the red line represents the orbit from $W^u(q)$ to $W^s(p)$, and the green line indicates the orbit from $W^u(p)$ to $W^s(q)$.} \label{examplefigure1}
\end{center}
\end{figure}

Then, by identifying $\mathbb{E}_p^{\si}$ and $\mathbb{E}_q^{\si}$ with a subspace for $\si=s,u$, one can take coordinates so that a neighborhood can be considered as a subset of $\mathbb{E}^s_p\times \mathbb{E}^u_p$ or $\mathbb{E}^s_q\times \mathbb{E}^u_q$, and the local stable and unstable manifolds are disks in the subspaces given by the splitting, $W^s_{loc}(p)\subset \mathbb{E}^s_p\times\{0\}$, $W^u_{loc}(p)\subset\{0\}\times \mathbb{E}^u_p$, $W^s_{loc}(q)\subset \mathbb{E}^s_q\times\{0\}$, and $W^u_{loc}(q)\subset\{0\}\times \mathbb{E}^u_q$.

Choose positive constants $\de^s_p$, $\de^u_p$, $\de^s_q$, $\de^u_q$, and set
\beqq
D^s_p:=W^s_{\de^s_p}(p),\ D^u_p:=W^u_{\de^u_p}(p),\ D^s_q:=W^s_{\de^s_q}(q),\ \mbox{and}\ D^u_q:=W^u_{\de^u_q}(q).
\eeqq
For convenience, suppose that $\Pi_p=D^s_p\times D^u_p$ and $\Pi_q=D^s_q\times D^u_q$.

{\bf (ii)} Consider the complex dynamics on the generalized ``heteroclinic" orbit joining the points $p$ and $q$.

By the assumptions on $\Upsilon_1$ and $\Upsilon_2$,
%there exist $p_0\in D^u_p\setminus\{p\}$, $q_0\in D^u_q\setminus\{q\}$, and a positive integer $k_1$, such that
%\beqq
%\phi(\tfrac{T}{2}+k_1T,p_0)\in\mbox{Int}(D^s_q)\ \mbox{and}\ \phi(\tfrac{T}{2}+k_1T,q_0)\in\mbox{Int}(D^s_p).
%\eeqq
%It is clear that, for any $k\geq k_1$,
%\beqq
%\phi(\tfrac{T}{2}+kT,p_0)\in\mbox{Int}(D^s_q)\ \mbox{and}\ \phi(\tfrac{T}{2}+kT,q_0)\in\mbox{Int}(D^s_p).
%\eeqq
there exist $p_0\in D^u_p\setminus\{p\}$, $q_0\in D^u_q\setminus\{q\}$, a positive integer $k_1\geq\max\{m_p,m_q\}$, and two constants $0<\eta^u_p<\de^u_p$ and $0<\eta^u_q<\de^u_q$, with
\beqq
\hat{H}^1_1:=D^s_p\times(p_0-\eta^u_p,p_0+\eta^u_p)\ \mbox{and}\ \hat{H}^2_1:=D^s_q\times(q_0-\eta^u_q,q_0+\eta^u_q),
\eeqq
where
$(p_0-\eta^u_p,p_0+\eta^u_p)\subset D^u_p\setminus\{p\}$ and $(q_0-\eta^u_q,q_0+\eta^u_q)\subset D^u_q\setminus\{q\}$, such that
\beqq
\phi(\tfrac{T}{2}+k_1 T,\hat{H}^1_1)\subset \Pi_q\ \mbox{and}\
\phi(\tfrac{T}{2}+k_1 T,\hat{H}^2_1)\subset \Pi_p.
\eeqq

It follows from the $\ld$-Lemma or the Inclination Lemma \cite[Lemma 7.1]{PalisMelo1982} that there is an integer $k_2\geq0$ such that
\beqq
D^u_q\subset\mbox{Proj}_{W^u(q)}(\phi(\tfrac{T}{2}+(k_1+k_2) T,\hat{H}^1_1)),
\eeqq
\beqq
D^u_p\subset\mbox{Proj}_{W^u(p)}(\phi(\tfrac{T}{2}+(k_1+k_2) T,\hat{H}^2_1)),
\eeqq
where $\mbox{Proj}_{W^u(p)}$ and $\mbox{Proj}_{W^u(q)}$ are the projections onto $W^u(p)$ and $W^u(q)$, respectively, and the assumptions that $\phi(\tfrac{T}{2}+m_pT,\Upsilon_1)$ contains a segment of $W^s_{loc}(q)$ and
$\phi(\tfrac{T}{2}+m_qT,\Upsilon_2)$ contains a segment of $W^s_{loc}(p)$ are used here.
Figure \ref{examplefigure-8} shows an illustrative diagram of the map from $\hat{H}^1_1\subset \Pi_p$ to $\Pi_q$.

\begin{figure}[h]
\begin{center}
\scalebox{0.5 }{ \includegraphics{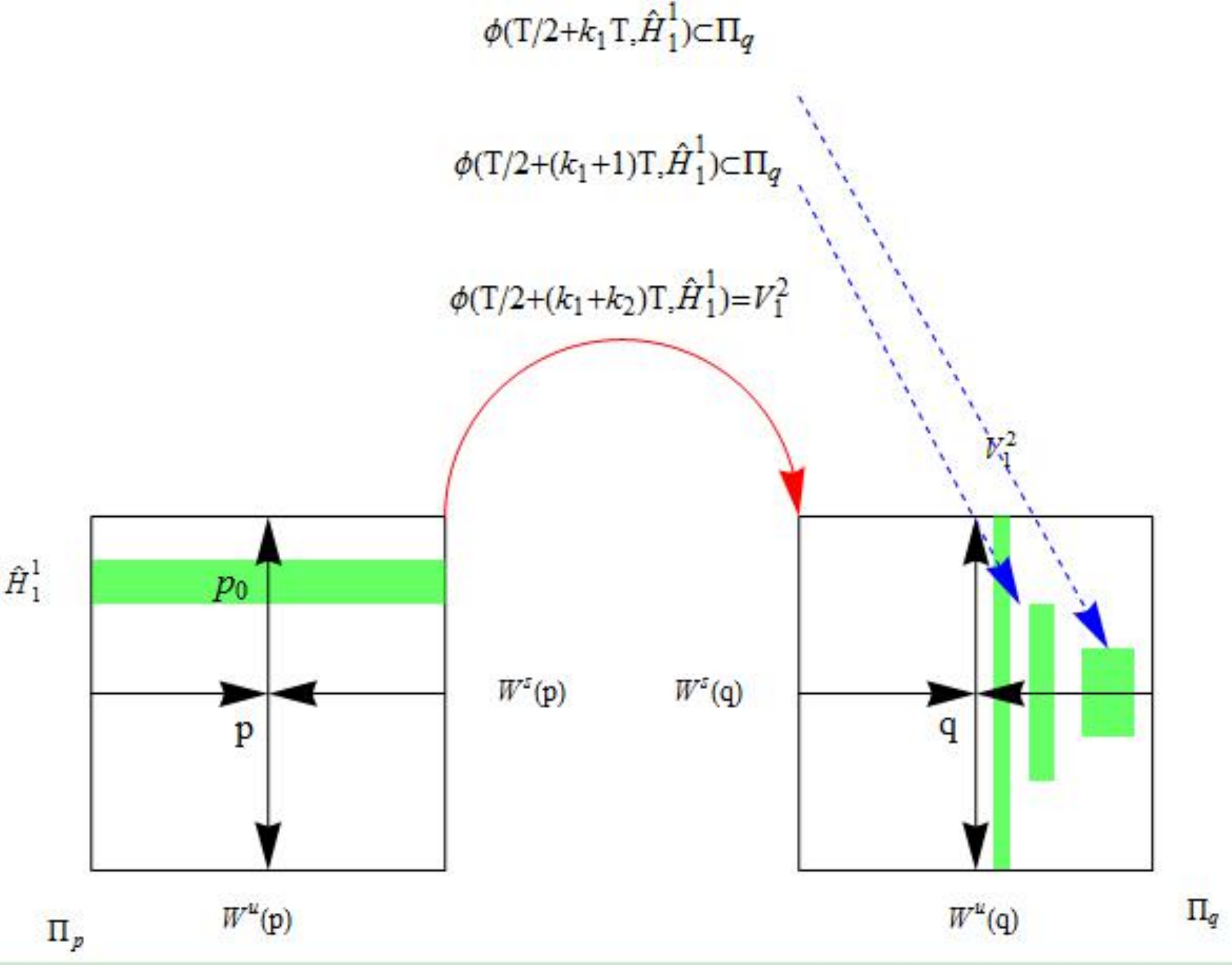}}
\renewcommand{\figure}{Fig.}
\caption{Illustrative diagram of the map from $\hat{H}^1_1\subset \Pi_p$ to $\Pi_q$.
}\label{examplefigure-8}
\end{center}
\end{figure}

{\bf (iii)} Consider the complex dynamics near the horizontal neighborhoods containing the points $p$ and $q$.

Since the periodic orbit is hyperbolic, it follows from the $\ld$-Lemma or Inclination Lemma \cite[Lemma 7.1]{PalisMelo1982} that there exists an integer $k_3\geq0$
such that, for any $k\geq k_3$, there exist positive constants $\ep^u_{p,k}$ and $\ep^u_{q,k}$, denote by
\beqq
H^1_{2,k}:=D^s_p\times(p-\ep^u_{p,k},p+\ep^u_{p,k})\ \mbox{and}\ H^2_{2,k}:=D^s_q\times(q-\ep^u_{q,k},q+\ep^u_{q,k}),
\eeqq
where
$(p-\ep^u_{p,k},p+\ep^u_{p,k})\subset D^u_p$ and $(p_0-\eta^u_{p},p_0+\eta^u_{p})\cap(p-\ep^u_{p,k},p+\ep^u_{p,k})=\emptyset$, and $(q-\ep^u_{q,k},q+\ep^u_{q,k})\subset D^u_q$ and $(q_0-\eta^u_{q},q_0+\eta^u_{q})\cap(q-\ep^u_{q,k},q+\ep^u_{q,k})=\emptyset$, such that
\beqq
D^u_p\subset\mbox{Proj}_{W^u(p)}(\phi(\tfrac{T}{2}+kT,H^2_{2,k})),
\eeqq
\beqq
D^u_q\subset\mbox{Proj}_{W^u(q)}(\phi(\tfrac{T}{2}+kT,H^1_{2,k})).
\eeqq

In the above discussion, one may assume that $\lim_{k\to+\infty}\ep^u_{p,k}=0$ and $\lim_{k\to+\infty}\ep^u_{q,k}=0$. An illustrative diagram is shown in Figure \ref{examplefig-5}. In this figure, in subgraph (a), the region bounded by red lines in $\Pi_p$ is $H^1_{2,k_3}$, the region bounded by green lines in $\Pi_p$ is $H^1_{2,k_3+1}$, the region bounded by blue lines in $\Pi_p$ is $H^1_{2,k_3+2}$; also in subgraph (a), the region bounded by red lines in $\Pi_q$ is $\phi(\tfrac{T}{2}+k_3T,H^1_{2,k_3})$, the region bounded by green lines in $\Pi_q$ is $\phi(\tfrac{T}{2}+(k_3+1)T, H^1_{2,k_3+1})$, the region bounded by blue lines in $\Pi_q$ is $\phi(\tfrac{T}{2}+(k_3+2)T,H^1_{2,k_3+2})$. In subgraph (b), the red rectangle in $\Pi_p$ is $H^1_{2,k_3}$, the green rectangle in $\Pi_p$ is $H^1_{2,k_3+1}$, the blue rectangle in $\Pi_p$ is $H^1_{2,k_3+2}$; also in subgraph (b), the red rectangle in $\Pi_q$ is $\phi(\tfrac{T}{2}+k_3T,H^1_{2,k_3})$, the green rectangle in $\Pi_q$ is $\phi(\tfrac{T}{2}+(k_3+1)T, H^1_{2,k_3+1})$, the blue rectangle in $\Pi_q$ is $\phi(\tfrac{T}{2}+(k_3+2)T,H^1_{2,k_3+2})$.

\begin{figure}[h]
\begin{center}
\subfloat[]{
\scalebox{0.25}{ \includegraphics{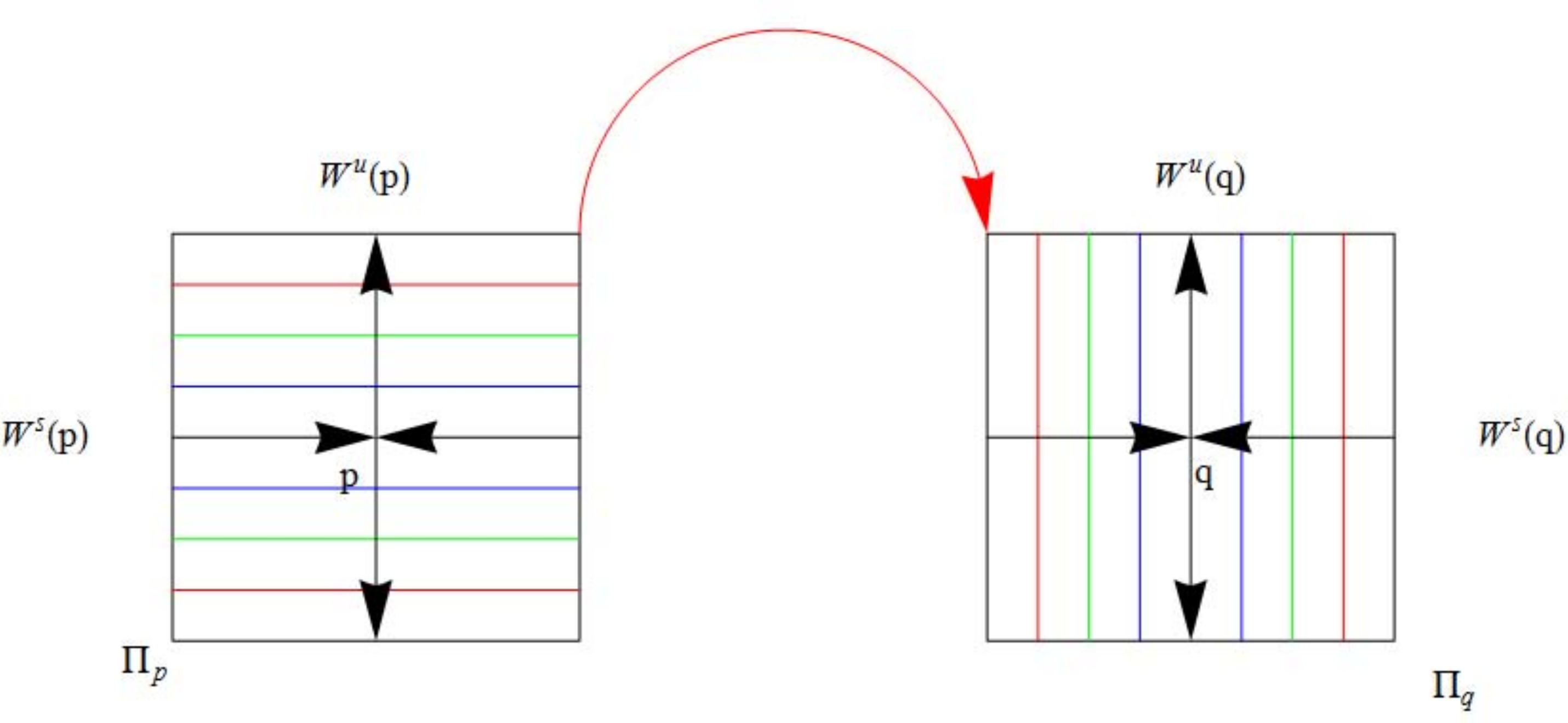}}}
\subfloat[]{
\scalebox{0.25}{ \includegraphics{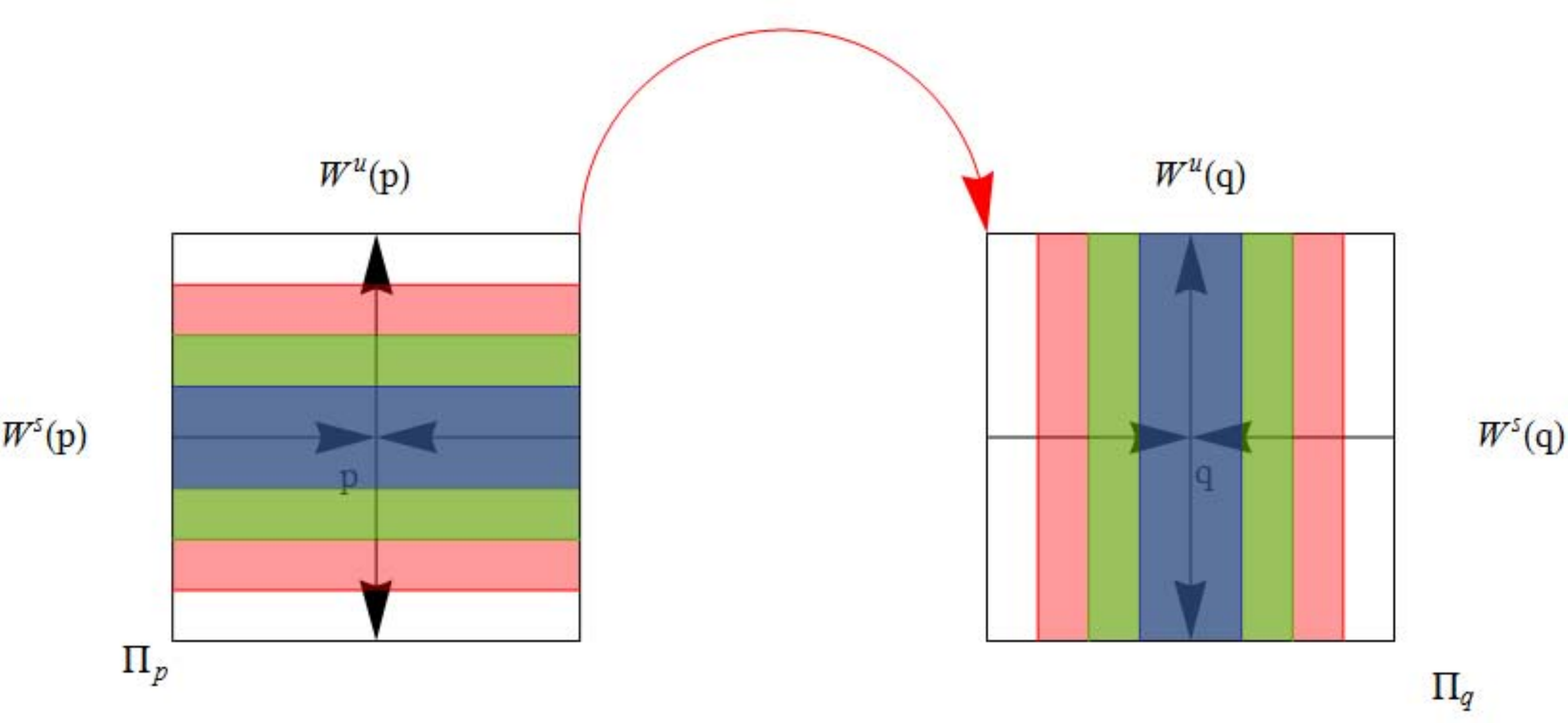}}}
\renewcommand{\figure}{Fig.}
\caption{Illustrative diagram for the complex dynamics near the horizontal neighborhoods containing the points $p$ and $q$. In subgraph (a), the region bounded by red lines in $\Pi_p$ is $H^1_{2,k_3}$, the region bounded by green lines in $\Pi_p$ is $H^1_{2,k_3+1}$, the region bounded by blue lines in $\Pi_p$ is $H^1_{2,k_3+2}$;
also in subgraph (a), the region bounded by red lines in $\Pi_q$ is $\phi(\tfrac{T}{2}+k_3T,H^1_{2,k_3})$, the region bounded by green lines in $\Pi_q$ is $\phi(\tfrac{T}{2}+(k_3+1)T, H^1_{2,k_3+1})$, the region bounded by blue lines in $\Pi_q$ is $\phi(\tfrac{T}{2}+(k_3+2)T,H^1_{2,k_3+2})$.
In subgraph (b), the red rectangle in $\Pi_p$ is $H^1_{2,k_3}$, the green rectangle in $\Pi_p$ is $H^1_{2,k_3+1}$, the blue rectangle in $\Pi_p$ is $H^1_{2,k_3+2}$;
also in subgraph (b), the red rectangle in $\Pi_q$ is $\phi(\tfrac{T}{2}+k_3T,H^1_{2,k_3})$, the green rectangle in $\Pi_q$ is $\phi(\tfrac{T}{2}+(k_3+1)T, H^1_{2,k_3+1})$, the blue rectangle in $\Pi_q$ is $\phi(\tfrac{T}{2}+(k_3+2)T,H^1_{2,k_3+2})$.
}\label{examplefig-5}
\end{center}
\end{figure}

{\bf (iv)} It is now ready to show the existence of a Smale horseshoe in a subshift of finite type with matrix $A$.

Take a sufficiently large $m\geq\max\{k_3,k_1+k_2\}$. Following the above discussions, by modifying some constants one can obtain: $0<\tilde{\eta}^u_p,\hat{\eta}^u_p\leq\eta^u_p$, $0<\tilde{\eta}^u_q,\hat{\eta}^u_q\leq\eta^u_q$, $0<\tilde{\ep}^u_{p,m},\hat{\ep}^u_{p,m}\leq\ep^u_{p,m}$, and
$0<\tilde{\ep}^u_{q,m},\hat{\ep}^u_{q,m}\leq\ep^u_{q,m}$. Set
\beqq
H^1_1:=D^s_p\times(p_0-\tilde{\eta}^u_p,p_0+\hat{\eta}^u_p)\ \mbox{and}\ H^2_1:=D^s_q\times(q_0-\tilde{\eta}^u_q,q_0+\hat{\eta}^u_q),
\eeqq
and
\beqq
H^1_{2}:=D^s_p\times(p-\tilde{\ep}^u_{p,m},p+\hat{\ep}^u_{p,m})\ \mbox{and}\ H^2_{2}:=D^s_q\times(q-\tilde{\ep}^u_{q,m},q+\hat{\ep}^u_{q,m}),
\eeqq
and
\beqq
V^2_1:=\phi(\tfrac{T}{2}+mT,H^1_1)\subset \Pi_q,\
V^2_2:=\phi(\tfrac{T}{2}+mT,H^1_2)\subset\Pi_q,
\eeqq
\beqq
V^1_1:=\phi(\tfrac{T}{2}+mT,H^2_1)\subset\Pi_p,\
V^1_2:=\phi(\tfrac{T}{2}+mT,H^2_2)\subset\Pi_p.
\eeqq
By the $\ld$-lemma or Inclination Lemma in \cite[Lemma 7.1]{PalisMelo1982}, for sufficiently large $m$, $H^i_j$ are $\mu_h$-horizontal strips, $1\leq i,j\leq2$, and $V^i_j$ are $\mu_v$-vertical strips, $1\leq i,j\leq2$, with $0\leq\mu_h\mu_v<1$. This, together with Theorem \ref{chaoshorsesho-1}, proves Theorem \ref{mainresult-1}.

An illustrative diagram for the map $\phi(\tfrac{T}{2}+mT,\cdot)$, which generates a Smale horseshoe in a subshift of finite type with matrix $A$, is shown in Figure \ref{examplefigure-7}.

\begin{figure}[h]
\begin{center}
\scalebox{0.5 }{ \includegraphics{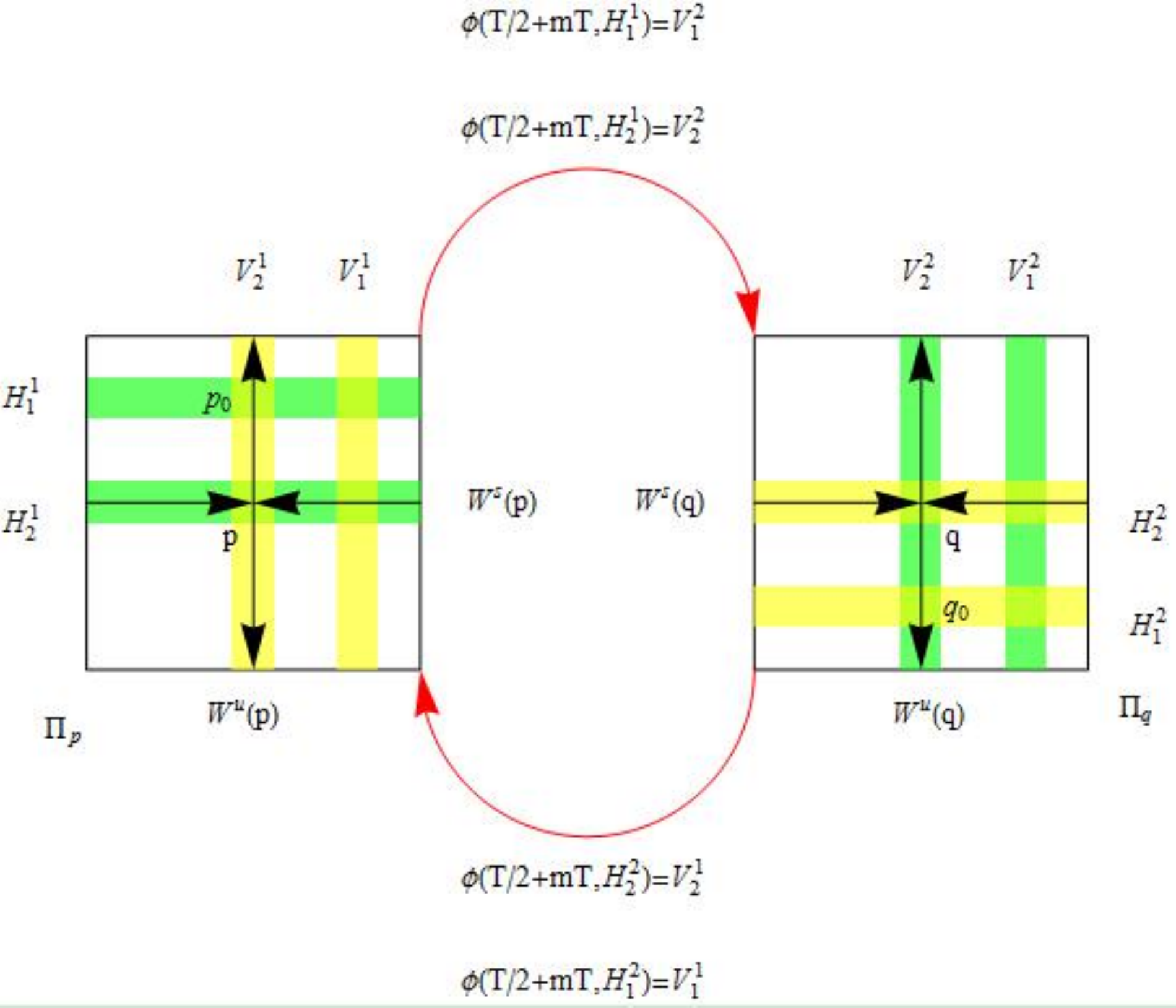}}
\renewcommand{\figure}{Fig.}
\caption{Illustrative diagram for the map $\phi(\tfrac{T}{2}+mT,\cdot)$, which generates a Smale horseshoe in a subshift of finite type with matrix $A$.
}\label{examplefigure-7}
\end{center}
\end{figure}

 In the above discussions, different assumptions bring various types of Smale horseshoes with respect to the Poincar\'{e} maps. An interesting question is it possible to obtain a classification of the continuous systems depending on the characterization of the chaotic dynamics by the existence of different types of Smale horseshoes?
\bigskip

Similar results could be obtained if there exist several periodic orbits, which might be used in the explanation of the complex dynamics of the multiscroll attractors \cite{LuChen2006}.

For simplicity, consider the situation with only two periodic orbits.

\begin{theorem}\label{mainresult-1}
Consider an ordinary differential equation, $\dot{x}=\Phi(x)$, $x\in\mathbb{R}^3$, where $\Phi:\mathbb{R}^3\to\mathbb{R}^3$ is differentiable.
Assume that
\begin{itemize}
\item there exist two hyperbolic periodic orbits of the same period $T>0$, denoted by $\Ga_1$ and $\Ga_2$ respectively, and points $p_i,q_i\in \Ga_i$, $i=1,2$, satisfying $\phi(\tfrac{T}{2},p_i)=q_i$ and $\phi(\tfrac{T}{2},q_i)=p_i$, $i=1,2$;
\item there are an open subset $\Upsilon_1\subset W^u_{loc}(p_1)\times W^s_{loc}(p_1)$ containing a line segment of $W^u_{loc}(p_1)$,  an open subset $\Upsilon_2\subset W^u_{loc}(q_2)\times W^s_{loc}(q_2)$ containing a line segment of $W^u_{loc}(q_2)$, and two positive integers $m_{p_1}$ and $m_{q_2}$ such that $\phi(\tfrac{T}{2}+m_{p_1}T,\Upsilon_1)$ contains a segment of $W^s_{loc}(p_2)$ with $\phi(\tfrac{T}{2}+m_{p_1}T,\Upsilon_1)\subset W^u_{loc}(p_2)\times W^s_{loc}(p_2)$, and $\phi(\tfrac{T}{2}+m_{q_2}T,\Upsilon_2)$ contains a segment of $W^s_{loc}(q_1)$ with $\phi(\tfrac{T}{2}+m_{q_2}T,\Upsilon_2)\subset W^u_{loc}(q_1)\times W^s_{loc}(q_1)$.
\end{itemize}

Then, there exist a positive integer $m$ and an invariant set $\Ld\subset\mathbb{R}^3$ such that the following relations hold:
\begin{center}
\begin{tikzpicture}
    % set up the nodes
    \node (E) at (0,0) {$\Ld$ };
    \node[right=of E] (F) at (4,0){$\Ld $};
    \node[below=of F] (A) {$\sum_8(B)$};
    \node[below=of E] (Asubt) {$\sum_8(B)$};
    % draw arrows and text between them
   \draw[->] (E)--(F) node [midway,above] {$\phi(\tfrac{T}{2}+mT,\cdot)$};
    \draw[->] (F)--(A) node [midway,right] {$\Psi$}
                node [midway,left] {};
    \draw[->] (Asubt)--(A) node [midway,below] {$\si$}
                node [midway,above] {};
    \draw[->] (E)--(Asubt) node [midway,left] {$\Psi$};
    %\draw[dashed] (Asubt)--(F);
\end{tikzpicture}
\end{center}
where $\Psi$ is a homeomorphism from $\Ld$ to $\sum_8(B)$, which is a topological conjugacy, and $B$ is a transition matrix:
$$\left(
  \begin{array}{cccccccc}
    0 & 0 & 1 & 1 & 1 & 1 & 0 & 0 \\
    0 & 0 & 1 & 1 & 1 & 1 & 0 & 0 \\
    1 & 1 & 0 & 0 & 0 & 0 & 0 & 0 \\
    1 & 1 & 0 & 0 & 0 & 0 & 0 & 0 \\
    0 & 0 & 0 & 0 & 0 & 0 & 1 & 1 \\
    0 & 0 & 0 & 0 & 0 & 0 & 1 & 1 \\
   0 & 0 & 1 & 1 & 1 & 1 & 0 & 0 \\
    0 & 0 & 1 & 1 & 1 & 1 & 0 & 0 \\
  \end{array}
\right).$$
\end{theorem}

\section*{Acknowledgements}

The authors would like to thank Prof. Qigui Yang and Dr. Yousu Huang for careful reading our manuscript and pointing out several mistakes, improving our presentation greatly.

This research was supported by the National Natural Science Foundation of China (No. 11701328).

\end{document}